\documentclass[reqno,11pt]{amsart}

\usepackage{amssymb,latexsym,amscd,amsthm,amssymb, amscd, amsfonts,enumerate,array,bbm,bm}
\usepackage[sorted,compressed-cites,sorted-cites,initials]{amsrefs}
\usepackage{youngtab}
\usepackage[cmtip,all]{xy}
\usepackage{geometry}
\usepackage{pifont,mathrsfs,mathtools,stmaryrd}

\newcommand{\ul}[1]{\underline{#1}}

\usepackage[english,francais]{babel}
\usepackage[utf8]{inputenc}

\usepackage[T1]{fontenc}

\usepackage[colorlinks]{hyperref}

\hypersetup{%
pdftitle={Generalized Joseph decomposition},
pdfauthor={},
pdfkeywords={},
bookmarksnumbered,
pdfstartview={FitH},
breaklinks=true,
linkcolor=blue,
urlcolor=blue,
citecolor=blue,
bookmarksdepth=2
}%

\usepackage{graphicx}

\numberwithin{equation}{section}

\newtheorem{theorem}{Theorem}[section]

\newtheorem{proposition}[theorem]{Proposition}

\newtheorem{corollary}[theorem]{Corollary}
\newtheorem{lemma}[theorem]{Lemma}
\renewcommand{\Im}{\operatorname{Im}}
\DeclareMathOperator{\ad}{ad}
\DeclareMathOperator{\ch}{ch}
\DeclareMathOperator{\mcat}{mod}

\DeclareMathOperator{\End}{End}
\DeclareMathOperator{\Iso}{Iso}

\DeclareMathOperator{\Hom}{Hom}
\DeclareMathOperator{\id}{id}

\DeclareMathOperator{\Mod}{Mod}
\newcommand{\la}{\langle}
\newcommand{\ra}{\rangle}

\theoremstyle{definition}

\newtheorem{example}[theorem]{Example}

\newcommand{\lact}{\triangleright}
\newcommand{\ract}{\triangleleft}

\newcommand{\lie}[1]{\mathfrak{#1}}
\newcommand{\tensor}{\otimes}

\DeclareFontFamily{U}{mathb}{\hyphenchar\font45}
\DeclareFontShape{U}{mathb}{m}{n}{
      <5> <6> <7> <8> <9> <10> gen * mathb
      <10.95> mathb10 <12> <14.4> <17.28> <20.74> <24.88> mathb12
      }{}
\DeclareSymbolFont{mathb}{U}{mathb}{m}{n}
\DeclareFontSubstitution{U}{mathb}{m}{n}
\DeclareMathSymbol{\smallsquare}        {2}{mathb}{"05}

\def\ZZ{\mathbb{Z}}
\def\QQ{\mathbb{Q}}

\def\gg{\mathfrak{g}}

\def\kk{\Bbbk}

\allowdisplaybreaks[2]

\begin{document}
\newgeometry{margin=1in}

\title[Generalized Joseph's decompositions]{Generalized Joseph's decompositions}
\author{Arkady Berenstein}
\address{Department of Mathematics, University of Oregon,
Eugene, OR 97403, USA} \email{arkadiy@math.uoregon.edu}

\author{Jacob Greenstein}
\address{Department of Mathematics, University of
California, Riverside, CA 92521.} 
 \email{jacob.greenstein@ucr.edu}

\thanks{The authors are partially supported by the NSF grant DMS-1403527 (A.~B) 
and by the Simons foundation collaboration grant no.~245735~(J.~G.).}

\selectlanguage{english}
\begin{abstract}
We generalize the decomposition of~$U_q(\gg)$ introduced by A.~Joseph in~\cite{joseph-mock} and relate it, for $\gg$~semisimple, 
to the celebrated computation of central elements due to V.~Drinfeld (\cite{Dr}). In that case we
construct a natural basis in
the center of $U_q(\gg)$ whose elements behave as Schur polynomials and thus explicitly identify the center
with the ring of symmetric functions.
\end{abstract}
\maketitle 

\section{Introduction and main results}  

\subsection{}
Let $H$ be an associative algebra with unity over a field~$\kk$ and let $\mathscr C$ be a full abelian subcategory closed under submodules of the category $H-\Mod$ of left $H$-modules.
Suppose that we have a ``finite duality'' functor ${}^\star: \mathscr C\to \Mod-H$ with $V^\star\subseteq V^*=\Hom_\kk(V,\kk)$ (with equality
if and only if~$V$ is finite dimensional) with its natural
right $H$-module structure, such that the restriction of the evaluation pairing $\la\cdot,\cdot\ra_V:V\tensor V^*\to \kk$ to $V\tensor V^\star$ is non-degenerate
for all objects $V$ in~$\mathscr C$ (see~\S\ref{subs:pf-main0} for the details).
Following~\cite{Fos}, we define $\beta_V:V\tensor_{D(V)} V^\star\to H^*$ where $D(V)=\End_H V^\star=(\End_H V)^{op}$ by 
$$
\beta_V(v\tensor f)(h)=\la h\lact v, f\ra_V=\la v,f\ract h\ra_V,\qquad v\in V,\,f\in V^\star,\,h\in H,
$$
where $\lact$ (respectively, $\ract$) denotes the left (respectively, right) $H$-action. It is easy to see that~$\beta_V$ is well-defined.
Set $H_V^*=\Im\beta_V$. 
Recall that $V\tensor V^\star$ and $H^*$ are naturally $H$-bimodules. 
The following is essentially proved in~\cite{Fos}*{\S3.1} and~\cite{FD}*{Corollary~1.16}
\begin{proposition}
\begin{enumerate}[{\rm(a)}]
\item\label{prop:Fos-0.a} For all $V\in \mathscr C$, $\beta_V$ is a homomorphism of $H$-bimodules and $H^*_V$ depends only
on the isomorphism class of~$V$. Moreover, if $V,V'\in\mathscr C$ are simple and $H^*_V=H^*_{V'}$ then $V\cong V'$;
\item\label{prop:Fos-0.b} 
 $H^*_{V\oplus V'}=H^*_V+H^*_{V'}$ for all $V,V'\in\mathscr C$. In particular, $H^*_{V^{\oplus n}}=H^*_V$ for all $n\in\mathbb N$.
\item\label{prop:Fos-0.c} 
If $V\tensor_{D(V)}V^\star$ is simple as an $H$-bimodule then $\beta_V$ is injective. 
\item\label{prop:Fos-0.d} If $V$ is simple finite dimensional then $V\tensor_{D(V)}V^\star$ is simple as an $H$-bimodule and hence $\beta_V$ is injective.
\end{enumerate}
\label{prop:Fos-0}
\end{proposition}

It is natural to call $H^*_V$ a {\em generalized Peter-Weyl component}. Denote $H^*_{\mathscr C}=\sum_{[V]\in\Iso\mathscr C} H^*_V$ and 
$\ul H^*_{\mathscr C}=\bigoplus_{[V]\in\Iso^\circ\mathscr C} H^*_V$,
where $\Iso\mathscr C$ (respectively, $\Iso^\circ\mathscr C$) is the set of isomorphism classes of objects (respectively, simple
objects) in~$\mathscr C$. By definition there is a natural homomorphism of $H$-bimodules $\ul H^*_{\mathscr C}\to H^*_{\mathscr C}$.
Clearly, under the assumptions of Proposition~\ref{prop:Fos-0}\eqref{prop:Fos-0.c} it is injective.
Note that $H^*_{\mathscr C}=\sum_{[V]\in A} H^*_V$ for any 
subset $A$ of~$\Iso\mathscr C$ which generates it as an additive monoid.
The following refinement of~\cite{Fos}*{Theorem~3.10} establishes the generalized Peter-Weyl decomposition.
\begin{theorem}\label{thm:semi-simp-new}
Suppose that all objects in~$\mathscr C$ have finite length. Then 
\begin{enumerate}[{\rm(a)}]
 \item if 
$H^*_{\mathscr C}=\ul H^*_{\mathscr C}$ then 
$\mathscr C$ is semisimple;
\item if $\mathscr C$ is semisimple and $V\tensor_{D(V)}V^\star$ is simple for every $V\in\mathscr C$ simple then $H^*_{\mathscr C}=\ul H^*_{\mathscr C}$.
\end{enumerate}
\end{theorem}

\subsection{}
Henceforth we denote by $\mathscr C^{fin}$ the full subcategory of~$\mathscr C$ consisting of all finite dimensional objects.
Clearly $V\tensor V^\star$, $V\in\mathscr C^{fin}$,
is a unital algebra with the unity~$1_V$; set
$z_V:=\beta_V(1_V)\in H^*_V$. For example, if $H=\kk G$ for 
a finite group $G$ then for any finite dimensional $H$-module~$V$ we have  
$z_V(g)=tr_V(g)$, $g\in G$ where 
$tr_V$ denotes the trace of a linear endomorphism of~$V$.

Given an $H$-bimodule~$B$, define the subspace $B^H$ of $H$-invariants in~$B$ by 
$B^H=\{ b\in B\,:\, h\lact b=b\ract h,\,\forall\, h\in H\}$ ($B^H$ is sometimes referred to as the center of~$B$). 
Clearly, $z_V\in (H^*_V)^H$, $z_V(1_H)=\dim_\kk V\not=0$ and $(H^*_V)^H=\kk z_V$ if $\End_H V=\kk\id_V$. 
Set $\mathcal Z_{\mathscr C}=\sum_{[V]\in\Iso\mathscr C} \ZZ z_V$. Given~$V\in\mathscr C$,
denote $|V|$ its image in the Grothendieck group $K_0(\mathscr C)$ of~$\mathscr C$. The following result contrasts sharply with 
Proposition~\ref{prop:Fos-0} and Theorem~\ref{thm:semi-simp-new} for non-semisimple $\mathscr C$.
\begin{theorem}\label{thm:char}
Suppose that $\mathscr C=\mathscr C^{fin}$. Then
the map $K_0(\mathscr C)\to \mathcal Z_{\mathscr C}$ given by
$|V|\mapsto z_V$, $[V]\in\Iso\mathscr C$ is an isomorphism  
of abelian groups.
\end{theorem}

\subsection{}
To introduce a multiplication on~$\mathcal Z_{\mathscr C}\subset (H^*_{\mathscr C})^H\subset H^*_{\mathscr C}$, 
we assume henceforth that $H=(H,m,\Delta,\varepsilon)$ is a bialgebra and that $\mathscr C$ is a tensor subcategory of $H-\Mod$. 
Note that $H^*$ is an algebra in a natural way.
It is easy to see (Lemma~\ref{lem:inv-subalg}) that $(H^*)^H$ is a subalgebra of~$H^*$.
We also assume that there is a natural isomorphism $(V\tensor V')^\star\cong V'{}^\star\tensor V^\star$ in~$\mcat-H$ 
for all $V,V'\in\mathscr C$.
\begin{theorem}
\begin{enumerate}[{\rm(a)}]
\item\label{thm:main-1.a} $H^*_V\cdot H^*_{V'}=H^*_{V\tensor V'}$ for all $V,V'\in\mathscr C$. In particular, $H^*_{\mathscr C}$ is a subalgebra of~$H^*$;
\item\label{thm:main-1.b} $z_V\cdot z_{V'}=z_{V\tensor V'}$ for all $V,V'\in\mathscr C^{fin}$. In particular, if 
$\mathscr C=\mathscr C^{fin}$ then $\mathcal Z_{\mathscr C}$ is a subring of~$(H_{\mathscr C}^*)^H$
and the map $K_0(\mathscr C)\to \mathcal Z_{\mathscr C}$ from Theorem~\ref{thm:char} 
is an isomorphism of rings.
\end{enumerate}
\label{thm:main-1}
\end{theorem}

\noindent
Thus, it is natural to regard $\mathcal Z_{\mathscr C}$ as the character ring of~$\mathscr C$.

\subsection{}
It turns out that we can transfer the above structures from~$H^*_{\mathscr C}$ to~$H$ if 
$H=(H,m,\Delta,\varepsilon,S)$ is a Hopf algebra.
For an $H$-bimodule $B$ define left actions $\ad$ and~$\diamond$ on~$B$ via 
$(\ad h)(b)=h_{(1)}\lact b\ract S(h_{(2)})$ and $h\diamond b=S^2(h_{(2)})\lact b\ract S(h_{(1)})$, $h\in H$,
$b\in B$, where $\Delta(b)=b_{(1)}\tensor b_{(2)}$ in Sweedler's notation.

Fix a categorical completion $H\widehat\tensor H$ such that $(f\tensor 1)(H\widehat\tensor H)\subset H$ for all $f\in H^*_{\mathscr C}$.
Equivalently, $\Phi_P:H^*_{\mathscr C}\to H$, $f\mapsto (f\tensor 1)(P)$ is 
a well-defined linear map.
Denote $\mathscr A(H)$ the set 
of all $P\in H\widehat\tensor H$ such that $P\cdot (S^2\tensor 1)(\Delta(h))=\Delta(h)\cdot P$ for all $h\in H$. 
Clearly, $\mathscr A(H)$ is a subalgebra of~$H\widehat\tensor H$. Elements of~$\mathscr A(H)$ are analogous to $M$-matrices
(see e.g.~\cite{Sh}).
For $V\in C^{fin}$, set $c_V=c_{V,P}:=\Phi_P(z_V)\in \Phi_P((H^*_{\mathscr C})^H)$. Let $Z(H)$ be the center of~$H$.
\begin{theorem}\label{thm:main-2}
Let~$P\in\mathscr A(H)$. Then 
$\Phi_P:H^*_{\mathscr C}\to H$ is a homomorphism of left $H$-modules, where $H$ acts on~$H^*_{\mathscr C}$ and~$H$ via $\diamond$ and~$\ad$, respectively.
Moreover, $\Phi_P((H^*_{\mathscr C})^H)\subset Z(H)$
and the assignment
$|V|\mapsto c_V$, $[V]\in\Iso\mathscr C^{fin}$ defines
a homomorphism of abelian groups $\ch_{\mathscr C}:K_0(\mathscr C^{fin})\to Z(H)$.
\end{theorem}
Surprisingly, $\Phi_P$ is often close to be an algebra homomorphism. To make this more precise,
we generalize the notion of an algebra 
homomorphism as follows. Let $A$, $B$ be $\kk$-algebras and let $\mathscr F$ be a collection of subspaces in~$A$. 
We say that a $\kk$-linear map $\Phi:A\to B$ is a {\em $\mathscr F$-homomorphism} if 
$\Phi(U)\cdot \Phi(U')\subset \Phi(U\cdot U')$ for all $U,U'\in\mathscr F$. We say that $\mathscr F$ is multiplicative 
if $U\cdot U'\in\mathscr F$ for all $U,U'\in\mathscr F$. It is easy to see that 
$|\mathscr F|:=\sum_{U\in\mathscr F} U$ is a subalgebra of~$A$ and $\Phi(|\mathscr F|)$ is a subalgebra of~$B$
for any multiplicative family $\mathscr F$.

In what follows we denote $\mathscr F_{\mathscr C}$ the collection of all subspaces of~$H^*$ of the form $H^*_V$ where 
$V\in\mathscr C$. By Theorem~\ref{thm:main-1}, $\mathscr F_{\mathscr C}$ is multiplicative.
\begin{example}
Let $H=\kk G$ where $G$ is a finite group and $\mathscr C$ be the category of its finite dimensional representations. Then the assignment $\delta_g\mapsto g^{-1}$ where $\delta_g(h)=\delta_{g,h}$, $g,h\in G$
defines an isomorphism of $H$-bimodules $\Phi:H^*\to H$. Let $\mathscr F_G=\{ H^*_V\,:\, [V]\in\Iso^\circ\mathscr C,\, \Hom_G(V,V\tensor V)\not=0\}\subset \mathscr F_{\mathscr C}$.
If $|G|\in\kk^\times$ then $\Phi$ is an $\mathscr F_G$-homomorphism
since $\Phi(H^*_V)\cdot \Phi(H^*_{V'})=0$ if $[V]\not=[V']\in\Iso^\circ\mathscr C$ and $\Phi(H^*_V)\cdot \Phi(H^*_V)=\Phi(H^*_V)$.
\end{example}

Denote by $\mathscr M(H)$ the set of all $P\in H\widehat\tensor H$ such that $\Phi_P$ is an $\mathscr F_{\mathscr C}$-homomorphism 
and by $\mathscr M_0(H)$ the set of all $P\in\mathscr M(H)$ such that $\Phi_P$ restricts to a homomorphism of algebras $(H^*_{\mathscr C})^H\to Z(H)$.
We abbreviate $H_{V,P}:=\Phi_P(H^*_V)$ and $H_{\mathscr C,P}:=\Phi_P(H^*_{\mathscr C})=\sum_{[V]\in\Iso\mathscr C} H_{V,P}$.
Since $\mathscr F_{\mathscr C}$ is multiplicative, $H_{\mathscr C,P}$ is a subalgebra of~$H$ for~$P\in\mathscr M(H)$.
The following is immediate.
\begin{proposition}\label{cor:main-cor}
Suppose that $P\in\mathscr A(H)\cap\mathscr M(H)$ and $\Phi_P$ is injective. Then:
\begin{enumerate}[{\rm(a)}]
\item\label{cor:main-cor.a'} If $V\tensor_{D(V)}V^\star$ is a simple $H$-bimodule then it is isomorphic to $H_{V,P}$ as a left $H$-module;
\item\label{cor:main-cor.a} $H_{\mathscr C,P}=\bigoplus_{[V]\in\Iso^\circ\mathscr C} H_{V,P}$ if $\mathscr C$ is semisimple and $V\tensor_{D(V)}V^\star$
is simple as an $H$-bimodule for each $V\in\mathscr C$ simple;
 \item\label{cor:main-cor.b} If $P\in\mathscr M_0(H)$ then 
 $\ch_{\mathscr C}:K_0(\mathscr C^{fin})\to Z(H)$ is injective.
 \end{enumerate}
\end{proposition}
The following theorem provides a sufficiently large subclass of~$\mathscr A(H)\cap\mathscr M(H)$ and~$\mathscr A(H)\cap\mathscr M_0(H)$.
\begin{theorem}\label{thm:main-thm}
Suppose that $P\in\mathscr A(H)$ such that $(\Delta\tensor 1)(P)=(m\tensor m\tensor 1)((T\tensor 1) P_{15}P_{35})$ for some $T\in H\widehat\tensor H\widehat\tensor H
\widehat\tensor H$. Then $P\in\mathscr M(H)$. Moreover, if $(m^{op}\tensor m^{op})(T)=1\tensor 1$ then $P\in\mathscr M_0(H)$.
%
\end{theorem}
It should be noted that $\mathscr M(H)$ and~$\mathscr M_0(H)$ are not exhausted by the above condition. 
\begin{example}\label{ex:S_3}
Suppose that $\operatorname{char}\kk\not=2,3$ and let 
$P_{\lambda,\mu}=\frac1{6}\sum_{\sigma\in S_3} 1\tensor \sigma+\frac1{36}\big[ s_1\tensor (1+(2\mu-1)s_1-(\mu+1)(s_2+s_1s_2s_1)+
s_1s_2+s_2 s_1)\big]_{S_3}+\frac1{18}\big[ s_1s_2\tensor (2+(\lambda-1)s_1s_2-(\lambda+1)s_2s_1)\big]_{S_3}$,
where $\lambda,\mu\in\kk$,
$s_i=(i,i+1)$ and we abbreviate $\big[ x\big]_G:=\sum_{g\in G} (g\tensor g)x(g^{-1}\tensor g^{-1})$ for $x\in\kk G\tensor\kk G$.
Then one can show that $P_{\lambda,\mu}\in\mathscr A(H)\cap\mathscr M_0(H)$ and that $\Phi_P$ is an isomorphism if and only if $(\lambda,\mu)\in(\kk^\times)^2$. However,
there is no $T\in H^{\tensor 4}$ such that the condition of Theorem~\ref{thm:main-thm} holds. 
\end{example}
It turns out that $P\in\mathscr A(\kk G)\cap \mathscr M_0(\kk G)$ with $\Phi_P$ injective does not always exist for a given finite group~$G$ 
(for instance, it does not exist for dihedral groups different from~$S_2\times S_2$ and~$S_3$) and thus it would be interesting to classify all finite 
groups~$G$ which admit such a~$P$. Its existence provides a decomposition of $\kk G$ into a direct sum of adjoint $G$-modules $H_{V,P}$ over all simple 
$\kk G$-modules~$V$ (a mock Peter-Weyl decomposition)
which is an alternative to the well-known Maschke decomposition into the direct sum of matrix algebras.
As a further example, we constructed an 8-parameter family of such~$P$
for $G=S_4$. The answer is rather cumbersome (it involves 34 terms of the form $[g\tensor x]_{S_4}$, $g\in S_4$, $x\in\kk S_4$) and is available at~\href{https://ishare.ucr.edu/jacobg/jdec-example.pdf}{https://ishare.ucr.edu/jacobg/jdec-example.pdf}).
%

Specializing Proposition~\ref{cor:main-cor} and Theorem~\ref{thm:main-thm} to quantized universal enveloping algebras we can recover 
Joseph's decomposition (\cite{joseph-mock}).
Namely, let $H=U_q(\gg)$ for a Kac-Moody algebra~$\gg$ and $\mathscr C_\gg$ be 
the (semisimple) category of highest weight integrable $U_q(\gg)$-modules (of type~$\mathbf 1$, see e.g. \cite{CP}); then $V^\star$ is 
the graded dual. Let $\Lambda^+$ be the monoid
of dominant weights for~$\gg$ and denote $V(\lambda)$ a highest weight simple integrable module of highest weight $\lambda\in\Lambda^+$.
We construct $P=P_{\gg}$ with $\Phi_{P_{\gg}}$ injective in Lemma~\ref{lem:P_g} and obtain the following Theorem which refines results of~\cite{joseph-mock}.
\begin{theorem}
\begin{enumerate}[{\rm(a)}]
\item\label{thm:joseph-decom.a} For $\lambda\in \Lambda^+$, $H_{V(\lambda),P}=\ad U_q(\gg)(K_{2\lambda})\cong V(\lambda)\tensor V(\lambda)^\star$.
\item\label{thm:joseph-decom.b} $\sum_{\lambda\in\Lambda^+} \ad U_q(\gg)(K_{2\lambda}) $ is direct and is a subalgebra of~$U_q(\gg)$.
\end{enumerate}
\label{thm:joseph-decomp}
\end{theorem}
Furthermore, part~\eqref{cor:main-cor.b} of~Proposition~\ref{cor:main-cor}, which generalizes a classic result of Drinfeld (\cite{Dr}), yields
\begin{theorem}\label{thm:centre}
Let $\gg$ be semisimple. Then the assignment $|V|\mapsto c_V$ defines an isomorphism of algebras $\QQ(q)\tensor_\ZZ K_0(\gg-\operatorname{mod})\to Z(U_q(\gg))$.
\end{theorem}
\noindent
This provides the following refinements of classic results of Duflo, Harish-Chandra and Rosso~(\cite{Ros}).
\begin{corollary}
For $\gg$ semisimple, 
$Z(U_q(\gg))$ is freely generated by the $c_{V(\omega)}$ where the $\omega$ are fundamental weights of~$\gg$, and 
$c_{V(\lambda)}c_{V(\mu)}=\sum_{\nu\in\Lambda^+} [V(\lambda)\tensor V(\mu):V(\nu)] c_{V(\nu)}$ for any $\lambda,\mu\in\Lambda^+$.
\end{corollary}

\subsection*{Acknowledgments}
We are grateful to Anthony Joseph for explaining to us his approach to the center of quantized enveloping algebras and
to Henning Andersen, David Kazhdan and Victor Ostrik  for stimulating discussions. 
This work 
was completed during
a visit of the second author to the Institut Mittag-Leffler (Djursholm, Sweden) whose support is greatly appreciated.

\section{Notation and proofs}

Recall that, given an $H$-bimodule $B$, $B^*$ is naturally an $H$-bimodule via $(h\lact f\ract h')(b)=f(h'\lact b\ract h)$, $f\in B^*$, $h,h'\in H$,
$b\in B$. In particular, $H^*$ is an $H$-bimodule.

\subsection{Proof of Theorem~\ref{thm:char}}\label{subs:pf-main0}
The following are immediate.
\begin{lemma}\label{lem:0}
$\la V,W^\star\ra_{V\oplus W}=0=\la W,V^\star\ra_{V\oplus W}$. 
\end{lemma}
\begin{lemma}\label{lem:A}
Let $V$, $W$ be left $H$-modules and let $\rho:H\tensor_\kk W\to V$ be a $\kk$-linear map. 
Then:
\begin{enumerate}[{\rm(a)}]
 \item \label{lem:A.a}
the assignment $
h\lact_\rho (v,w)=(h\lact v+\rho(h\tensor w),h\lact w)$, $h\in H$, $v\in V$, $w\in W$,
defines a left $H$-module structure $V\oplus_\rho W$ on~$V\oplus W$ if and only if
\begin{equation}\label{eq:lem:A.1}
\rho(hh'\tensor w)=\rho(h\tensor h'\lact w)+h\lact \rho(h'\tensor w),\qquad h,h'\in H,\,w\in W.
\end{equation}
In that case $V$ is an $H$-submodule of $V\oplus_\rho W$ and $W=(V\oplus_\rho W)/V$.
\item \label{lem:A.b}
A short exact sequence of $H$-modules $0\to V\to U\to W\to 0$ is equivalent to 
$0\to V\xrightarrow{} V\oplus_\rho W\xrightarrow{} W\to 0$ for some~$\rho$
satisfying~\eqref{eq:lem:A.1}.
\end{enumerate}
\end{lemma}
Thus, given $V\subset U$ in~$\mathscr C$, we can replace the natural short exact sequence $0\to V\to U\to U/V\to 0$ by
the one from Lemma~\ref{lem:A}. 
\begin{lemma}\label{lem:B'}
Let $V$, $W$ be left $H$-modules and let $\rho$ be as in Lemma~\ref{lem:A}. 
Then $\beta_{V\oplus_\rho W}(x+y)=\beta_V(x)+\beta_V(y)$
for any $x\in V\tensor V^\star$, $y\in W\tensor W^\star$.
\end{lemma}
\begin{proof}
It suffices to verify the assertion for $x=v\tensor f$ and $y=w\tensor g$, $v\in V$, $w\in W$,
$f\in V^\star$, $g\in W^\star$. We have by Lemmata~\ref{lem:0}, \ref{lem:A}\eqref{lem:A.a}
\begin{align*}
\beta_{V\oplus_\rho W}&(v\tensor f+w\tensor g)(h)=\la h\lact_\rho v\tensor f+h\lact_\rho w\tensor g\ra_{V\oplus W}
\\&=\la h\lact v,f\ra_V+\la\rho(h\tensor w), f\ra_{V\oplus W}+\la h\lact w,g\ra_{W}=\beta_V(v\tensor f)(h)+\beta_{W}(w\tensor g)(h).\qedhere
\end{align*}
\end{proof}
Since $1_{V\oplus_\rho W}=1_V+1_W$ where $1_V\in V\tensor V^\star$, $1_W\in W\tensor W^\star$, it follows 
from Lemma~\ref{lem:B'} that $z_{V\oplus_\rho W}=z_V+z_W$ and the map $K_0(\mathscr C)\to \mathcal Z_{\mathscr C}$, $|V|\mapsto z_V$ is a well-defined
surjective homomorphism of abelian groups.
Also, $z_V\in \sum_{[S]\in \Iso^\circ\mathscr C} \ZZ z_{S}$
for each $V\in\mathscr C=\mathscr C^{fin}$ because it has finite length.
Since the set $\{z_{V}\}_{[V]\in \Iso^\circ\mathscr C}\subset \ul H^*_{\mathscr C}$ is $\kk$-linearly independent by Proposition~\ref{prop:Fos-0}\eqref{prop:Fos-0.d},
the injectivity follows.\qed
%

\subsection{Algebra structure on~\texorpdfstring{$H^*_{\mathscr C}$}{H*\_C}}
Henceforth we assume that $H=(H,m,\Delta,\varepsilon)$ is a bialgebra. Then $H^*$ is a unital algebra with the multiplication defined by 
$(\phi\cdot\xi)(h)=\phi(h_{(1)})\xi(h_{(2)})$, $h\in H$, $\phi,\xi\in H^*$, $\Delta(h)=h_{(1)}\tensor h_{(2)}$ in Sweedler notation and the unity is~$\varepsilon$.
\begin{lemma}\label{lem:inv-subalg}
$(H^*)^H$ is a subalgebra of~$H^*$.
\end{lemma}
\begin{proof}
Observe that $\phi\in (H^*)^H$ if and only if $\phi(hh')=\phi(h'h)$ for all~$h,h'\in H$. Then, given $h,h'\in H$ and $\xi,\xi'\in(H^*)^H$ we have 
\begin{equation*}
(\xi\cdot\xi')(hh')=\xi(h_{(1)}h'_{(1)})\xi'(h_{(2)}h'_{(2)})=\xi(h'_{(1)}h_{(1)})\xi'(h'_{(2)}h_{(2)})=(\xi\cdot\xi')(h'h).\qedhere
\end{equation*}
\end{proof}

\begin{proof}[Proof of Theorem~\ref{thm:main-1}] 
Note that in the category of $\kk$-vector spaces there is a natural isomorphism 
$\kappa:(V\tensor V^\star)\tensor (V'\tensor V'{}^\star)\to(V\tensor V')\tensor (V\tensor V')^\star$, 
$\kappa(v\tensor f\tensor v'\tensor f')=v\tensor v'\tensor f'\tensor f$, $v\in V$, $v'\in V'$,
$f\in V^\star$, $f'\in V'{}^\star$. Then, clearly,
$\la\cdot,\cdot\ra_{V\tensor V'}\circ\kappa=\la \cdot,\cdot\ra_V\tensor \la \cdot,\cdot\ra_{V'}$ which immediately implies that
$\tilde\beta_V\tensor \tilde\beta_{V'}=\tilde\beta_{V\tensor V'}\circ \kappa$ where $\tilde\beta_U:=\beta_U\circ\pi_U$ where 
$\pi_U:U\tensor_\kk U^\star\to U\tensor_{D(U)} U^\star$ is the natural projection.
This proves the first assertion and also the second once we observe that $1_{V\tensor V'}=\kappa(1_V\tensor 1_{V'})$.
\end{proof}

\subsection{The Hopf algebra case}
Suppose now that $H=(H,m,\Delta,\varepsilon,S)$ is a Hopf algebra. Since~$H$ is naturally an $H$-bimodule,
$\ad:H\to\End_\kk H$ is a homomorphism of algebras. We also define $\ad^*:H^{op}\to \End_\kk H$
by $(\ad^*h)(h')=S(h_{(1)})h'S^2(h_{(2)})$, which is a homomorphism of algebras. Henceforth, given 
$a\in H^{\tensor n}$ we write it in Sweedler-like notation as $a=a_1\tensor\cdots\tensor a_n$ with summation understood.
\begin{proof}[Proof of Theorem~\ref{thm:main-2}]
We need the following equivalent descriptions of~$\mathscr A(H)$.
\begin{lemma}\label{lem:D}
Let $P=P_1\tensor P_2\in H\widehat\tensor H$. The following are equivalent:
\begin{enumerate}[{\rm(a)}]
 \item\label{lem:D.a} $P\cdot (S^2\tensor 1)\circ\Delta(h)=\Delta(h)\cdot P$;
 \item\label{lem:D.c} $(1\tensor h)\cdot P=(\ad^*h_{(1)})(P_1)\tensor P_2 h_{(2)}$;
 \item\label{lem:D.d} $(\ad^* h\tensor 1)(P)=(1\tensor\ad h)(P)$.
\end{enumerate}
\end{lemma}
\begin{proof}
By~\eqref{lem:D.a} we have
$
h_{(1)}\tensor P_1 S^2(h_{(2)})\tensor P_2 h_{(3)}\tensor h_{(4)}=h_{(1)}\tensor h_{(2)}P_1\tensor h_{(3)}P_2\tensor h_{(4)}
$
for all $h\in H$.
Then~\eqref{lem:D.c} and~\eqref{lem:D.d} follow by applying 
$m(S\tensor 1)\tensor 1\tensor \varepsilon$ and $m(S\tensor 1)\tensor m(1\tensor S)$, respectively, to both sides. Part~\eqref{lem:D.c}
implies~\eqref{lem:D.a} since
$h_{(1)}(\ad^* h_{(2)})(h')=h'S^2(h)$.
Finally, \eqref{lem:D.d} implies~\eqref{lem:D.c} since 
$(\ad^*h_{(1)})(P_1)\tensor P_2h_{(2)}=P_1\tensor \ad h_{(1)}(P_2)h_{(2)}=P_1\tensor hP_2$.
\end{proof}
\begin{lemma}\label{lem:E}
Let $B$ be an $H$-bimodule and set $B^{\diamond H}:=\{ b\in B\,:\, h\diamond b=\varepsilon(h)b,\,h\in H\}$.
Then $B^H\subset B^{\diamond H}\subset B^{S(H)}$ with the equality if $S$ is invertible.
\end{lemma}
\begin{proof}
Let $h\in H$. Then for all $b\in B^H$ 
we have $h\diamond b=S^2(h_{(2)})\lact b\ract S(h_{(1)})=S^2(h_{(2)})S(h_{(1)})\lact b=S(h_{(1)}S(h_{(2)}))\lact b=
\varepsilon(h)b$. On the other hand, for all $b\in H^{\diamond H}$,
$
S(h)\lact b=\varepsilon(h_{(1)})S(h_{(2)})\lact m=S(h_{(3)})S^2(h_{(2)})\lact m\ract S(h_{(1)})=S(S(h_{(2)})h_{(3)})\lact m\ract S(h_{(1)})=m\ract S(h)
$.
\end{proof}
The following Lemma is well-known and can be proved similarly.
\begin{lemma}\label{lem:F}
$Z(H)=H^H=H^{\ad H}:=\{ h'\in H\,:\, (\ad h)(h')=\varepsilon(h)h',\,h\in H\}$.\qed
\end{lemma}
By Lemma~\ref{lem:D}\eqref{lem:D.d} we have, for all $h\in H$, $\xi\in H^*_{\mathscr C}$
$$
\Phi_P(h\diamond \xi)=(S^2(h_{(2)})\lact\xi\ract S(h_{(1)}))(P_1)P_2=\xi((\ad^*h)P_1)P_2=\xi(P_1)(\ad h)(P_2)=(\ad h)\Phi_P(\xi).
$$
Furthermore, if $\xi\in (H^*_{\mathscr C})^H$ then $\Phi_P(h\diamond \xi)=\varepsilon(h)\Phi_P(\xi)=(\ad h)\Phi_P(\xi)$,
whence $\Phi_P(\xi)\in Z(H)$.
\end{proof}
\begin{proof}[Proof of Theorem~\ref{thm:main-thm}]
Suppose that $P$ satisfies 
$(\Delta\tensor 1)(P)=t_1 P_1 t_2\tensor t_3 P'_1 t_4\tensor P_2 P'_2$, for some $T=t_1\tensor t_2\tensor t_3\tensor t_4\in H^{\widehat\tensor 4}$
where $P=P_1\tensor P_2=P'_1\tensor P'_2$. 
Then for any $\xi,\xi'\in H^*_{\mathscr C}$
\begin{equation}
\begin{split}
\Phi_P(\xi\cdot \xi')=(\xi\cdot \xi')(P_1)P_2=\xi( t_1P_1t_2)\xi'(t_3P'_1t_4)P_2P'_2=(t_2\lact \xi\ract t_1)(P_1)(t_4\lact \xi'\ract t_3)(P'_1)P_2P'_2
\\=\Phi_P(t_2\lact \xi\ract t_1)\cdot\Phi_P(t_4\lact \xi'\ract t_3).\label{eq:Phi-mult}
\end{split}
\end{equation}
Take $\xi\in H^*_V$, $\xi'\in H^*_{V'}$. Then $\xi\cdot \xi'\in H^*_{V\tensor V'}$ by Theorem~\ref{thm:main-1}\eqref{thm:main-1.a} and 
$\Phi_P(\xi\cdot\xi')\in H_{V,P}\cdot H_{V',P}$ by~\eqref{eq:Phi-mult}. Therefore, $P\in\mathscr M(H)$. 
Furthermore, assume that $t_2t_1\tensor t_4t_3=1\tensor 1$, and let 
$\xi,\xi'\in (H^*_{\mathscr C})^H$. Then~\eqref{eq:Phi-mult} yields 
$\Phi_P(\xi\cdot \xi')=\Phi_P(t_2t_1\lact \xi)\cdot\Phi_P(t_4t_3\lact \xi')=\Phi_P(\xi)\cdot\Phi_P(\xi')$.
This implies that~$P\in\mathscr M_0(H)$.
\end{proof}


\subsection{Applications}\label{subs:Joseph}
Let $\mathscr R(H)$ be the set of pairs $(R^+,R^-)$, $R^\pm\in H\widehat\tensor H$, such that $R^+_{21}R^-\cdot\Delta(h)=\Delta(h)\cdot R^+_{21}R^-$ 
for all~$h\in H$ and
$
(\Delta\tensor 1)(R^\pm)=R^\pm_{13}R^\pm_{23}$, $(1\tensor\Delta)(R^+)=R^+_{13}R^+_{12}
$.
Clearly, $(R,R)\in\mathscr R(H)$ if $R$ is an $R$-matrix for~$H$.
\begin{lemma}\label{lem:P-R}
Suppose that there exists $\mathbf g\in H$ group-like such that $\mathbf gS^2(h)=h\mathbf g$ for all~$h\in H$. 
Let $(R^+,R^-)\in\mathscr R(H)$.
Then $P:=R^+_{21}\cdot R^-\cdot (\mathbf g\tensor 1)\in\mathscr A(H)\cap \mathscr M_0(H)$.
\end{lemma}
\begin{proof}
Write $R^\pm=r^\pm_1\tensor r^\pm_2=s^\pm_1\tensor s^\pm_2$.
Since $R^+_{21}R^-\cdot\Delta(h)=\Delta(h)\cdot R^+_{21}R^-$ 
we have $$P\cdot (S^2\tensor 1)(\Delta(h))=r^+_2r^-_1\mathbf g S^2(h_{(1)})\tensor r^+_1r^-_2 h_{(2)}=r^+_2r^-_1 h_{(1)}\mathbf g\tensor r^+_1r^-_2 h_{(2)}
=\Delta(h)\cdot P.$$ Thus, $P\in \mathscr A(H)$. Furthermore, 
$(\Delta\tensor 1)(P)=R^+_{32}R^+_{31}R^-_{13}R^-_{23}(\mathbf g\tensor \mathbf g\tensor 1)=
P_1\tensor r^+_2r^-_1\mathbf g\tensor r^+_1 P_2r^-_2$. 
Since $(\Delta\tensor 1)(R^+)=
r^+_1\tensor s^+_1\tensor r^+_1s^+_1$, by Lemma~\ref{lem:D}\eqref{lem:D.c} we obtain 
\begin{multline*}
(\Delta\tensor 1)(P)=(\ad^* r^+_1)(P_1)\tensor r^+_2s^+_2r^-_1\mathbf g\tensor P_2 s^+_1r^-_2
=(\ad^* r^+_1)(P_1)\tensor r^+_2 P'_1\tensor P_2P'_2\\
=S(r^+_1)P_1S^2(s^+_1)\tensor r^+_2s^+_2 P'_1\tensor P_2P'_2.
\end{multline*}
Thus, $P\in\mathscr M(H)$ with $T=(S\tensor S^2\tensor 1\tensor 1)(R^+_{13}\cdot R^+_{23})$. Finally, 
$(m^{op}\tensor m^{op})(T)=
S^2(s^+_2)S(r^+_1)\tensor r^+_2s^+_2=(S\tensor 1)( R^+\cdot (S\tensor 1)(R^+))=1\tensor 1$.
Thus, $P\in\mathscr M_0(H)$.
\end{proof}
If~$P$ is as in Lemma~\ref{lem:P-R} we obtain 
\begin{equation}\label{eq:explicit-phi}
\Phi_P(\beta_V(v\tensor f))= r^+_1\la r_2^+r^-_1\mathbf g\lact v,f\ra_V r^-_2=r^+_1\la r^-_1\lact\mathbf g(v),f\ract r^+_2\ra_V r^-_2,\qquad 
v\in V,\, f\in V^\star.
\end{equation}

Let $\kk=\QQ(q)$ and let 
$U_q(\gg)$ be a quantized enveloping algebra corresponding to a 
symmetrizable Kac-Moody algebra~$\gg$ which is a Hopf algebra generated by $E_i$, $F_i$, $i\in I$ and $K_\mu$, $\mu\in\Lambda$, where $\Lambda$ is 
a weight lattice of~$\gg$, with $\Delta(E_i)=1\tensor E_i+E_i\tensor K_{\alpha_i}$, $\Delta(F_i)=F_i\tensor 1+K_{-\alpha_i}\tensor F_i$,
$\Delta(K_\mu)=K_\mu\tensor K_\mu$, $\varepsilon(E_i)=\varepsilon(F_i)=0$ and $\varepsilon(K_\mu)=1$, where $\alpha_i$, $i\in I$ are simple roots of~$\gg$.
Let $\mathcal K$ be the subalgebra of~$U_q(\gg)$ generated by the $K_\mu$, $\mu\in\Lambda$.
After~\cites{Dr,Lus-book}, there exists a unique $R$-matrix in a weight completion $U_q(\gg)\widehat\tensor U_q(\gg)$ of the form 
$R=R_0 R_1$ where $R_1\in U_q^+(\gg)\widehat\tensor U_q^-(\gg)$ is essentially $\Theta^{op}$ in the notation of~\cite{Lus-book} and satisfies
$(\varepsilon\tensor 1)(R_1)=(1\tensor \varepsilon)(R_1)=1\tensor 1$, while 
$R_0\in \mathcal K\widehat\tensor \mathcal K$ is determined by
the following condition: for any $\mathcal K$-modules
$V^\pm$ such that $K_\mu|_{V^\pm}=q^{(\mu,\mu_\pm)}\id_{V^\pm}$, $\mu,\mu_\pm\in\Lambda$, we have
$R_0|_{V^-\tensor V^+}=q^{(\mu_-,\mu_+)}\id_{V^-\tensor V^+}$. Here $(\cdot,\cdot)$ is the Kac-Killing form on $\Lambda\times\Lambda$
(\cite{Kac}). The following 
is immediate.
\begin{lemma}\label{lem:P_g}
Let $R=r_1\tensor r_2$ be as above. Let $v_\lambda\in V(\lambda)$ ($f_\lambda\in V(\lambda)^\star$) be a highest (respectively, lowest) 
weight vector of weight~$\lambda$ 
(respectively, $-\lambda$), $\lambda\in\Lambda^+$. Then $r_1\lact v_\lambda\tensor r_2=v_\lambda\tensor K_\lambda$ and $r_1\tensor f_\lambda\ract r_2=
K_\lambda\tensor f_\lambda$.\qed
\end{lemma}
\begin{proof}[Proof of Theorem~\ref{thm:joseph-decomp}]
Since $V(\lambda)$ is a simple highest weight module, $D(V(\lambda))\cong \kk$. Note that for any $\lambda,\mu\in\Lambda^+$, $V(\lambda)\tensor V(\mu)$ is a 
simple $U_q(\lie g\oplus\lie g)=U_q(\lie g)\tensor U_q(\lie g)$-module of highest weight $(\lambda,\mu)$. Twisting $V(\mu)$ with the anti-automorphism 
of $U_q(\lie g)$ interchanging $F_i$ and~$E_i$, we conclude that $V(\lambda)\tensor V(\lambda)^\star$ is a simple $U_q(\lie g)$-bimodule.
Taking into account that $\mathbf g=K_{-2\rho}$ we obtain from Lemma~\ref{lem:P_g} and~\eqref{eq:explicit-phi} that $\Phi_P(\beta_{V(\lambda)}(v_\lambda\tensor f_\lambda))=
K_\lambda\la \mathbf g\lact v_\lambda,f_\lambda\ra K_\lambda\in \kk^\times K_{2\lambda}$. Since $V(\lambda)\tensor V(\lambda)^\star$ is cyclic on~$v_\lambda\tensor f_\lambda$ as $U_q(\gg)$-module with the $\diamond$ action, $H_{V(\lambda)}$ is cyclic on~$K_{2\lambda}$ as the $\ad U_q(\gg)$-module by the above.
Since $\beta_{V(\lambda)}$ is injective by Theorem~\ref{prop:Fos-0}\eqref{prop:Fos-0.c} and $\Phi_P$ is injective by~\cite{Dr}, it follows that $H_{V(\lambda)}\cong V(\lambda)\tensor V(\lambda)^\star$. This proves~\eqref{thm:joseph-decom.a}. Then the sum in~\eqref{thm:joseph-decom.b} is direct by 
Proposition~\ref{cor:main-cor}\eqref{cor:main-cor.a} and coincides with $H_{\mathscr C_{\lie g},P}$
which is always a subalgebra of~$H$.
\end{proof}

\begin{proof}[Proof of Theorem~\ref{thm:centre}]
Since~$D(V(\lambda))\cong\kk$, Theorem~\ref{thm:joseph-decomp} implies that 
$Z(H_{\mathscr C_\gg,P_\gg})=\bigoplus_{\lambda\in\Lambda^+} \kk c_{V(\lambda)}$, hence the assignment $|V(\lambda)|\mapsto c_{V(\lambda)}$
is an  isomorphism
$\kk\tensor_\ZZ K_0(\mathscr C_\gg)\to \Phi_{P_\gg}((H^*_{\mathscr C_\gg})^H)=Z(H_{\mathscr C_\gg,P_{\gg}})$
as in Proposition~\ref{cor:main-cor}\eqref{cor:main-cor.b}.
By~\cite{Lus},
$K_0(\mathscr C_\gg)=K_0(\gg-\operatorname{mod})$ where $\gg-\operatorname{mod}$
is the category of finite dimensional $\gg$-modules. 
On the other hand, each non-zero element of $Z(U_q(\gg))$ is $\ad$-invariant, hence generates a one-dimensional
$\ad U_q(\gg)$-module and thus is 
contained in $H_{\mathscr C_\gg,P_\gg}$ 
by~\cite{joseph-mock}. Therefore, $Z(U_q(\gg))\subset H_{\mathscr C_\gg,P_\gg}$ hence $Z(U_q(\gg))=Z(H_{\mathscr C_\gg,P_\gg})$.
\end{proof}

\end{document}